  \documentclass[a4paper,11pt]{article}
  \usepackage{amsmath,amsthm,amsfonts,amssymb}
  \newtheorem{thm}{Theorem}[section]
  \newtheorem{lem}[thm]{Lemma}

  \newtheorem*{thmA}{Theorem A}
  \newtheorem*{thmB}{Theorem B}
  \newtheorem*{thmC}{Theorem C}

  \newtheorem*{pro*}{Problem}
  \newtheorem*{thm1.2}{Theorem}
\newcommand{\FF}{{\mathbb{F}}}

\newcommand{\Syl}{{\operatorname{Syl}}}

  \title{Finite groups whose maximal subgroups \\of order divisible by all the primes are supersolvable}
  \author{}
  \date{}

\begin{document}
  \maketitle

 \bigskip
  \centerline{by}
  \bigskip

 \smallskip
  \centerline{Alexander Moret\'o}  \centerline{Departament
  de Matem\`atiques} \centerline{Universitat de Val\`encia}
  \centerline{46100 Burjassot. Val\`encia SPAIN} \centerline{  Alexander.Moreto@uv.es}

 \vskip 10pt

{\bf Abstract.}  We study finite groups $G$ with the property that for any subgroup $M$ maximal in $G$ whose order is divisible by all the prime divisors of $|G|$, $M$ is supersolvable. We show that any nonabelian simple group can occur as a composition factor of such a group and that, if $G$ is solvable, then the nilpotency length and the rank are arbitrarily large. On the other hand, for every prime $p$, the $p$-length of such a group is at most $1$. This answers questions proposed by V. Monakhov in {\it The Kourovka Notebook}.

{\bf AMS Subject Classification.}  20D10, 20F16

{\bf Keywords and phrases.} supersolvable subgroup, maximal subgroup, simple group, solvable group, $p$-length, Fitting height
  \vfill

  \noindent   Research  supported by Ministerio de Ciencia e Innovaci\'on PID-2019-103854GB-100, FEDER funds  and Generalitat Valenciana AICO/2020/298.

\section{Introduction}

Problem 19.55 in {\it The Kourovka Notebook} \cite{kou}, proposed by V. Monakhov,  asks the following.

\begin{pro*}
Suppose that in a finite group $G$ every maximal subgroup $M$ is supersolvable whenever $\pi(M)=\pi(G)$, where $\pi(G)$ is the set of all prime divisors of the order of $G$.
\begin{enumerate}
\item
What are the nonabelian composition factors of $G$?
\item
Determine exact upper bounds for the nilpotency length, the $p$-length and the rank of $G$ if $G$ is solvable.
\end{enumerate}
\end{pro*}

The goal of this note is to answer these questions. First, we show that any nonabelian simple group can occur as a composition factor of some group with this property. It is perhaps remarkable that we do not need the classification of finite simple groups to prove this.

\begin{thmA}
Every nonabelian simple group can occur as a composition factor of a finite group with the property that every maximal subgroup $M$ of $G$ with $\pi(M)=\pi(G)$ is supersolvable.
\end{thmA}

Regarding the second part of Monakhov's question, we show that there is not any bound for nilpotency length and the rank,  but the $p$-length is at most $1$ for every prime $p$.

\begin{thmB}
There exist solvable groups of arbitrarily large nilpotency length and arbitrarily large rank with the property that  every maximal subgroup $M$ of $G$ with $\pi(M)=\pi(G)$ is supersolvable.
\end{thmB}

\begin{thmC}
Let $G$ be a solvable group such that for every maximal subgroup $M$ of $G$ such that $\pi(M)=\pi(G)$, $M$ is supersolvable. Then the $p$-length of $G$ is $1$ for every prime divisor $p$ of $|G|$.
\end{thmC}

Note that groups of $p$-length at most $1$ for every prime $p$ are known to have  a number of properties (see, for instance, VI.6 of \cite{hup}).

\section{Arbitrary groups}

In this section we prove Theorem A. We start by recalling two well-known lemmas.
  
  \begin{lem}
 \label{nor} 
  Let $G$ be a finite group and $N\trianglelefteq G$. If $P$ is a Sylow $p$-subgroup of $G$,  then $N_{G/N}(PN/N)=N_G(P)N/N$.
  \end{lem}
  
  \begin{lem}
\label{siz} 
  Let $A_p$ be the alternating group on $p$ letters, where $p$ is a prime. Let $P\in\Syl_p(A_p)$. Then $|N_{A_p}(P)|=\frac{p-1}{2}p$.
  \end{lem}
  
  Now, we prove Theorem A.
  
  \begin{proof}[Proof of Theorem A]
  Let $q$ be bigger than  the largest prime divisor of $|S|$ and let $p$ be a prime such that $p>2q$. Note that this implies that $\frac{p-1}{2}\geq q$.
  Let $G=S\wr H$, where $H=A_p$ permutes transitively $p$ copies of $S$. Let $P\in\Syl_p(H)$ so that $P$ is also a Sylow $p$-subgroup of $G$. Let $N=S\times\cdots\times S$ be the base group. By Lemmas \ref{nor} and \ref{siz},
  $$
  |N_G(P)N/N|=|N_{G/N}(PN/N)|=|N_{A_p}(P)|=\frac{p-1}{2}p.
  $$
  This implies that if $r$ is a prime such that $q\leq\frac{p-1}{2}<r<p$, then $r$ does not divide $|N_G(P)|$. Notice that by Bertrand's Postulate, such a prime exists. 
  
  Now, we will show that $G$ does not have any maximal supersolvable subgroup of order divisible by all the primes in $\pi(G)$. By way of contradiction, let $M$ be such a maximal supersolvable subgroup. Since $p$ divides $|M|$, we may assume that $P\leq M$. Since $M$ is supersolvable, $p$ is the largest prime divisor of $|M|$, and $P\in\Syl_p(M)$, $P\trianglelefteq M$. But then $M\leq N_G(P)$, so all the prime divisors of $|G|$ divide $|N_G(P)|$. This is a contradiction.
  \end{proof}
  
  \section{Solvable groups}

 We start with the proof of Theorem B. The key to our construction is the following well-known lemma.

  \begin{lem}
  \label{cons}
  Let $G$ be a finite (complex) linear group of degree $n$.  Let $p$ be a prime such that $e=\exp(G)$ divides $p-1$.  Then $G$ acts faithfully and irreducibly  on an elementary abelian group $V$ of order $p^n$.   
\end{lem}

\begin{proof}
Since  $\FF_p$ contains a primitive $\exp(G)$th root of unity, Corollary 9.15 of \cite{isa} implies that $\FF_p$ is a splitting field for $G$.  By hypothesis, $G$ has a faithful (complex) irreducible character $\chi$ of degree $n$ and since $p$ does not divide $|G|$,  Theorem 15.13 of \cite{isa} implies that $\chi$ is also an irreducible  $p$-Brauer character. Since  $\FF_p$ is a splitting field for $G$, the natural module for this character is an elementary abelian group of order $p^n$. The result follows.
\end{proof}

\begin{proof}[Proof of Theorem B]
Now, let $G_1=V_1$ be a cyclic group of order $p_1$ for some prime $p_1$. By Lemma \ref{cons}, $G_1$ acts faithfully and irreducibly on  an elementary abelian $p_2$-group $V_2$ (of order $p_2$, in this case). Put $G_2=G_1\ltimes V_2$. As before, $G_2$ acts faithfully and irreducibly on some elementary abelian $p_3$-group $H_3$ and we put $G_3=G_2\ltimes V_3$. Inductively, we define $G_{n+1}=G_n\ltimes V_{n+1}=V_1\dots V_{n+1}$ for $n\geq1$. Notice that since $V_n$ is the unique minimal normal subgroup of $G_n$ for every $n\geq 1$, $G_n$ has faithful irreducible characters and these groups do exist by Lemma \ref{cons}.

Note that the nilpotency length of $G_n$ is $n$ for every $n$. Also, the rank of $V_n$ goes to infinity when $n$ goes to infinity (for instance, because the nilpotency length of a linear group over a finite field is bounded in terms of the dimension. See Theorem 3.9(b) of \cite{mw}).  

It remains to see that the groups $G_n$ satisfy the hypothesis of the question. Clearly, we may assume $n>2$.  Note that $|G_n|_{p_i}=p_i$ if and only if $i\leq 2$. Thus, if $M$ is a maximal subgroup of $G$ such that  $\pi(M)=\pi(G)$, then $|M|_{\{p_1,p_2\}}=p_1p_2=|G|_{\{p_1,p_2\}}$. Thus if $p_j$ is the prime divisor of $|G:M|$, $j\geq3$.  Put $N_k=V_k\dots V_{n+1}$ for every $k$ so that $N_{j+1}<M\cap N_j<N_j$ (the first inequality is strict because $p_j$ divides $|M|$). Put $H_j=V_1\dots V_{j-1}$ so that $G=H_j\ltimes N_j$. Set $\pi=\{p_1,\dots,p_{j-1}\}$ and notice that $|H_j|=|M|_{\pi}$. Let $H$ be a Hall $\pi$-subgroup of $M$ and note that $H$ and $H_j$ are conjugate so  $G=H\ltimes N_j$ and the action of $H$ on $N_j/N_{j+1}$ is irreducible. But $M\cap N_j$ is $H$-invariant. This is a contradiction. This means that $\pi(M)\neq\pi(G)$ for every $M$ maximal in $G$, and we are done. 
\end{proof}

Finally, we prove that the $p$-length is bounded.



\begin{proof}[Proof of Theorem C]
Notice that the hypothesis is inherited by quotients. Thus, if $G$ is a minimal counterexample, we may assume that $l_p(G)=2$ but $l_p(G/N)=1$ for every nontrivial normal subgroup $N$ of $G$. By VI.6.9 of \cite{hup}, for instance,  $O_{p'}(G)=1$, $V=O_p(G)$ is elementary abelian  and is the unique minimal normal subgroup of $G$ and $G=HV$ for some subgroup $H$. Since $l_p(G)=2$, $p$ divides $|H|$ so by hypothesis $H$ is supersolvable. 
Notice that $O_p(H)=1$. Let $L=O_{p'}(H)$, $K/L=O_p(H/L)$ and let $U$ be a Hall $p'$-subgroup of $H$. Notice that $\pi(UV)=\pi(G)$ so by hypothesis $LV\leq UV$ is supersolvable. Write $V=V_1\times\cdots\times V_t$ with $V_i$ cyclic of order $p$ for every $i$ and $L$-invariant. Thus $L$ is isomorphic to a subgroup of the direct product of $t$ copies of the cyclic group of order $p-1$. In particular, all prime divisors of $|L|$ are less than $p$. 
Now, since $H$ is supersolvable and $K/L$ acts faithfully on $L$, we have a contradiction.
\end{proof}

\end{document}